\documentclass[12pt]{amsart}
\usepackage[left=3cm,top=2cm,right=3cm]{geometry}
\usepackage{fancyhdr}
\usepackage{bbm}
\usepackage{mathrsfs}
\usepackage[usenames,dvipsnames,svgnames,table]{xcolor}
\usepackage[pdftex,pagebackref,colorlinks=true,linkcolor=BrickRed,citecolor=OliveGreen,pdfstartview=FitH]{hyperref}
\usepackage{amssymb}

\usepackage{graphicx}
\usepackage{hyperref}
\usepackage{float}
\usepackage[normalem]{ulem}
\usepackage{verbatim}

\makeatletter
\newsavebox{\@brx}
\newcommand{\llangle}[1][]{\savebox{\@brx}{\(\m@th{#1\langle}\)}%
  \mathopen{\copy\@brx\kern-0.5\wd\@brx\usebox{\@brx}}}
\newcommand{\rrangle}[1][]{\savebox{\@brx}{\(\m@th{#1\rangle}\)}%
  \mathclose{\copy\@brx\kern-0.5\wd\@brx\usebox{\@brx}}}
\makeatother

\newtheorem{thm}{Theorem}[section]
\newtheorem{ass}[thm]{Assumption}
\newtheorem{coro}[thm]{Corollary}
\newtheorem{lem}[thm]{Lemma}
\newtheorem{prop}[thm]{Proposition}

\theoremstyle{definition}

\newtheorem{eg}[thm]{Example}

\newtheorem{nota}[thm]{Notation}

\theoremstyle{remark}

\newtheorem{remk}[thm]{Remark}

\renewcommand{\phi}{\varphi}

\newcommand{\Rbb}{ {\mathbb R}}
\newcommand{\Zbb}{ {\mathbb Z}}

\newcommand{\Cbb}{ {\mathbb C}}

\newcommand{\be}{\begin{enumerate}}
\newcommand{\ee}{\end{enumerate}}

\newcommand{\Oe}{\Omega_{\epsilon}}

\newcommand{\geps}{g_{\epsilon}}

\title{Geometric inequalities between Dirichlet and Neumann eigenvalues}
\author{Lawford Hatcher}

\begin{document}

\begin{abstract}
    Comparing Neumann and Dirichlet eigenvalues of the Laplacian on a bounded domain $\Omega\subseteq\Rbb^n$ is a topic that goes back at least to the work of P\'olya \cite{polya}. We study the effect of the isoperimetric ratio of $\Omega$ on the number $N(\Omega)$ of Neumann eigenvalues that do not exceed the first Dirichlet eigenvalue, proving that $N(\Omega)$ is bounded above and below by a constant multiple of the isoperimetric ratio in the case of convex domains. We also show that these estimates do not hold in the non-convex setting, addressing questions of Cox-MacLachlan-Steeves \cite{coxetal} and Freitas \cite{freitas}. Despite these counterexamples, we find similar estimates for polygonal domains in $\Rbb^2$ as well as certain families of fiber bundles that asymptotically collapse onto their base spaces, the motivating examples being tubular neighborhoods of submanifolds. 
\end{abstract}

\maketitle

\section{Introduction}\label{intro}
Let $\Omega$ be a Lipschitz domain in $\Rbb^n$ or a compact Riemannian $n$-manifold with nonempty boundary. Denote by $0=\mu_1<\mu_2\leq \mu_3\leq \dots $ (resp. $0<\lambda_1<\lambda_2\leq \lambda_3\leq \dots $) the eigenvalues of the Laplacian on $\Omega$ with Neumann (resp. Dirichlet) boundary conditions, counted with multiplicity. Let $$N(\Omega):=\#\{j\mid \mu_j\leq \lambda_1\}.$$
\indent Motivated by results of Cox-MacLachlan-Steeves \cite{coxetal} and Freitas \cite{freitas}, the primary goal of this paper is to examine the relationship between the isoperimetric ratio of $\Omega$ and $N(\Omega)$. Inequalities involving $N(\Omega)$ date back at least to the work of P\'olya \cite{polya}, in which he proves that $N(\Omega)\geq 2$ for all bounded Lipschitz planar domains $\Omega\subseteq \Rbb^2$. This result has been extended to inequalities between higher eigenvalues, higher dimensions, and curved spaces by various authors: \cite{payne}, \cite{levinewein}, \cite{fried}, \cite{ashbaugh}, \cite{filonov}, \cite{rohleder}, \cite{franketal}, \cite{huaetal}. 

\begin{nota}
    Throughout the paper, we deal with sets of various Hausdorff dimensions, and the dimension will always be clear in context. Given a set $C\subseteq\Rbb^n$ of Hausdorff dimension $k$, we will denote the $k$-dimensional Hausdorff measure of $C$ by $|C|$. For example, $|\partial\Omega|$ denotes the $(n-1)$-dimensional Hausdorff measure of the boundary of $\Omega$ while $|\Omega|$ denotes the $n$-dimensional Hausdorff measure of $\Omega$.
\end{nota}

\subsection{Bounded convex domains} 
Our first main result is a two-sided bound on $N(\Omega)$ for bounded convex domains $\Omega\subseteq \Rbb^n$.

\begin{thm}\label{mainthmconvex}
    For each $n\geq 2$, there exist constants $C_1,C_2>0$ depending only on $n$ such that for any bounded convex domain $\Omega\subseteq \Rbb^n$, $$C_1\cdot\frac{|\partial\Omega|^n}{|\Omega|^{n-1}}\leq N(\Omega)\leq C_2\cdot \frac{|\partial\Omega|^n}{|\Omega|^{n-1}}.$$ Moreover, it suffices to take $$C_1=\omega_n\cdot 184^{-n}\cdot n^{-\frac{3}{2}n(n+1)}\;\;\text{and}\;\;C_2=92^n\cdot n^{\frac{3}{2}n^2+n},$$ where $\omega_n$ is the volume of the unit ball in $\Rbb^n$. 
\end{thm}
We make no claim that the constants $C_1$ and $C_2$ in Theorem \ref{mainthmconvex} are sharp.  However, it would be interesting to find the sharp constants, and we now describe a potential improvement. Let $E$ be the constant in Theorem \ref{funanoestimate} (Funano's main theorem in \cite{funano}) below. Then we may take $$C_1=\omega_n\cdot 2^{-n}\cdot E^{-n/2}\cdot n^{-\frac{3}{2}n(n+1)}\;\;\text{and}\;\;C_2=E^{n/2}\cdot n^{\frac{3}{2}n^2+n}.$$ Therefore, improving the constant in Funano's theorem automatically improves our constant. Though some effort has been made by Freitas and Kennedy \cite{freitaskennedy} to study the optimal constant, its exact value is still unknown. \\
\indent Due to the convexity hypothesis, Theorem \ref{mainthmconvex} implies another interesting set of geometric bounds on $N(\Omega)$. Let $\mathcal{V}_{n-1}$ denote the set of all $(n-1)$-dimensional affine subspaces of $\Rbb^n$, and set $$A_{n-1}(\Omega):=\sup_{V\in \mathcal{V}_{n-1}}|\Omega\cap V|.$$ In other words, $A_{n-1}(\Omega)$ is the volume of the largest $(n-1)$-dimensional cross section of $\Omega$. It is clear that $A_{n-1}(\Omega)\leq \frac12|\partial\Omega|$, and we show in Lemma \ref{crosssectionperimeter} below that $$|\partial\Omega|\leq \frac{(2n)^n}{\omega_{n-1}}A_{n-1}(\Omega).$$ We expect that this inequality is not new, but we have been unable to find it in the literature. These observations prove that
\begin{coro}\label{crosssectionest}
    Let $n\geq 2$, and let $C_1$ and $C_2$ be as in Theorem \ref{mainthmconvex}. Then for any bounded convex domain $\Omega\subseteq \Rbb^n$, $$2^n\cdot C_1\cdot \frac{A_{n-1}(\Omega)^n}{|\Omega|^{n-1}}\leq N(\Omega)\leq \frac{(2n)^{n^2}}{\omega_{n-1}^{n}}\cdot C_2\cdot\frac{A_{n-1}(\Omega)^n}{|\Omega|^{n-1}}.$$
\end{coro}

\subsection{Non-convex domains}
\indent In recent works, the authors make conjectures similar to the statement of Theorem \ref{mainthmconvex}. In particular, Conjecture 2 in \cite{coxetal} asserts that the conclusion of Theorem \ref{mainthmconvex} holds for \textit{any} bounded Lipschitz domain in $\Rbb^n$ (possibly with different constants $C_1$ and $C_2$), and Conjecture 2 of \cite{freitas} implies that for $n=2$, $$\frac{|\partial\Omega|}{\sqrt{\pi|\Omega|}}\leq N(\Omega).$$ In Section \ref{counterexsection}, we present a sequence of counterexamples showing that each of these conjectures is false for non-convex domains. One may wonder whether, instead, the conclusion of Corollary \ref{crosssectionest} holds for all bounded Lipschitz domains. The counterexamples we give in Section \ref{counterexsection} show that this also is not the case. \\
\indent Despite this sequence of counterexamples, the numerical experiments of Cox-MacLachlin-Steeves indicate that for general non-pathological domains in $\Rbb^n$, there should be some relationship between the isoperimetric ratio and $N(\Omega)$. Unfortunately, we are unable to completely characterize this relationship in this paper. However, we do prove similar bounds to Theorem \ref{mainthmconvex} in the setting of polygonal domains.\\
\indent For $m\geq 3$ an integer, let $\mathcal{P}_m$ denote the set of all polygons with at most $m$ edges. To be precise, we define a \textit{polygon} to be a bounded Lipschitz domain $P\subseteq \Rbb^2$ whose boundary is a finite union of line segments. Note that we do not require that a polygon be simply connected. We allow for vertices to have angle equal to $\pi$, so $\mathcal{P}_m\subseteq \mathcal{P}_n$ for $m\leq n$.
\begin{thm}\label{polygonthm}
    For all $m\geq 3$, there exist $C_1,C_2>0$ depending only on $m$ such that for any $P\in \mathcal{P}_m$, $$C_1\cdot \frac{|\partial P|}{\sqrt{|P|}}\leq N(P)\leq C_2\cdot \frac{|\partial P|^2}{|P|}.$$ Moreover, we may take $$C_1=\frac{j_0}{4\sqrt{2\pi}\cdot m(m-1)}\;\;\text{and}\;\; C_2=55(m-2)^2,$$ where $j_0\approx 2.4048$ is the first positive zero of the Bessel function of order zero. 
\end{thm}
\begin{remk}
    The lower bound in Theorem \ref{polygonthm} is weaker than the one in Theorem \ref{mainthmconvex}. Since triangles are convex, Theorem \ref{mainthmconvex} implies that this lower bound is not optimal if $m=3$. In contrast, for fixed $m\geq 4$, we show in Example \ref{polygoncounterex} below that the lower bound for polygons is actually sharp up to the constant $C_1$. By Theorem \ref{mainthmconvex}, the upper bound in Theorem \ref{polygonthm} is sharp up to the constant $C_2$. We again make no claim that the constants themselves are sharp.
\end{remk}

\subsection{Fiber bundles}
We also study $N(\Omega)$ for certain families of Riemannian manifolds that degenerate onto lower dimensional submanifolds. Recall that on a Riemannian manifold $(\Omega,g_{\Omega})$, the Riemannian metric defines the Laplace-Beltrami operator. If $\Omega$ has nonempty boundary, then we can define Neumann and Dirichlet Laplace-Beltrami operators. We use the same notation for the eigenvalues of these operators as in the Euclidean case, and we will use the notation $|\cdot|$ to denote the Riemannian volume of a manifold.\\
\indent Because the exact statement is somewhat technical, we put off stating the general theorem we prove until Section \ref{mfldsection}. Roughly speaking, we find a precise asymptotic on $N(\Omega)$ in terms of the isoperimetric ratio for a family of fiber bundles with shrinking nearly isometric fibers. See Theorem \ref{asymptoticthm} below. Here we illustrate the utility of Theorem \ref{asymptoticthm} with a simple example. Let $B^n$ denote the unit Euclidean $n$-ball, and let $S^{n-1}:=\partial B^n$ denote the unit sphere. 
\begin{thm}\label{tubenbhd}
    Let $(X,g_X)$ be an $(n+m)$-dimensional Riemannian manifold, and let $M\subseteq X$ be a closed $m$-dimensional submanifold. For $\epsilon>0$, let $\Oe$ denote the set of points in $X$ with geodesic distance at most $\epsilon$ to $M$. Then, as $\epsilon\to 0^+$, 
    \begin{align*}
        N(\Oe)&=\frac{\omega_m}{(2\pi)^m}\cdot|M|\cdot\Bigg(\sum_{j=1}^{N(B^n)}\big(\lambda_1(B^n)-\mu_j(B^n)\big)^{m/2}\Bigg)\cdot\frac{1}{\epsilon^m}+o\Big(\frac{1}{\epsilon^m}\Big)\\
        &=\frac{\omega_m\omega_n^{n+m-1}}{(2\pi)^m|S^{n-1}|^{n+m}}\cdot\Bigg(\sum_{j=1}^{N(B^n)}\big(\lambda_1(B^n)-\mu_j(B^n)\big)^{m/2}\Bigg)\cdot \frac{|\partial\Oe|^{n+m}}{|\Oe|^{n+m-1}}+o\Big(\frac{|\partial\Oe|^{n+m}}{|\Oe|^{n+m-1}}\Big).
    \end{align*}
\end{thm}
\begin{remk}
    The asymptotic provided by Theorem \ref{tubenbhd} depends only on the dimensions of $X$ and $M$ and the volume of $M$ but not on the particular isometric embedding of $M$ in $X$ or the geometry of $X$. 
\end{remk}
\begin{remk}
    The asymptotic behavior of eigenvalues on domains and manifolds that degenerate onto lower dimensional objects has been long studied. For a small selection of related results, see \cite{fukaya}, \cite{fs}, \cite{freitas2}, and \cite{luc1}. See also the survey paper \cite{grieser}.
\end{remk}

\subsection{Outlines of proofs}
\indent We now outline the proofs of the main theorems, beginning with Theorem \ref{mainthmconvex}.\\
\indent We prove in Proposition \ref{existrectangles} that given a convex domain $\Omega$, after applying an isometry if necessary, there exists a rectangular prism $R\subseteq \Omega$ such that $\Omega$ is contained in a rectangular prism $Q$ obtained by dilating $R$ by $n^{3/2}$. As a result, the isoperimetric ratio of $\Omega$ is bounded above and below by constant multiples of the isoperimetric ratio of $R$, where the constants depend only on $n$. In the Appendix, we prove a ``quantitative" Neumann version of Dirichlet domain monotonicity, slightly upgrading a result of Funano \cite{funano} (namely, we remove the piecewise smoothness assumption on $\partial\Omega$):
\begin{thm}[cf. Theorem 1.1 \cite{funano}]\label{funanoestimate}
    Let $n\geq 2$. There exists a universal constant $E>0$ such that if $\Omega\subseteq\Omega'\subseteq\Rbb^n$ are bounded convex domains, then for all $j\geq 1$, $$\mu_j(\Omega')\leq E\cdot n^2 \cdot \mu_j(\Omega).$$ Moreover, we make take $E=92^2$.
\end{thm}
\noindent Applying Theorem \ref{funanoestimate} to the inclusions $R\subseteq\Omega\subseteq Q$, we reduce the problem to two simple lattice point counting problems similar to the ones solved in the proof of Theorem 2 of \cite{coxetal}, and this yields the estimates we desire in Theorem \ref{mainthmconvex}.\\
\indent The lower and upper bounds in Theorem \ref{polygonthm} are proved with different methods. The lower bound follows from the observation that the longest edge of an $m$-gon gives a trivial upper bound on the total perimeter. We then construct a test function supported on a strip of $\Rbb^2$ that is orthogonal to the longest edge and combine this with the Faber-Krahn inequality to get the lower bound. The upper bound follows from a decomposition scheme that partitions a non-convex polygon into a union of convex polygons. We then apply an argument based on Neumann domain monotonicity and the Payne-Weinberger inequality \cite{payneweinberger} to get the inequality. \\
\indent To prove Theorem \ref{asymptoticthm} below, from which Theorem \ref{tubenbhd} follows, we first consider the case of trivial Riemannian product manifolds, in which case the eigenvalues are computable in terms of the constituents of the product. As the product manifolds collapse onto a lower dimensional submanifold, we apply Weyl's law to find an asymptotic on $N(\Omega)$. To get the general result, we study perturbations of the product metric and then use a Dirichlet-Neumann bracketing argument to get a global result applicable to fiber bundles. \\
We now set some notation that will be employed throughout the paper. 
\begin{nota}
    When we wish to emphasize the dependence of the eigenvalues on the domain or manifold $\Omega$, we may write $\mu_j(\Omega)$ (resp. $\lambda_j(\Omega)$).
\end{nota}
\begin{nota}
    Suppose that $y\in \Rbb^n$, $\lambda\in \Rbb$, and $X\subseteq \Rbb^n$. By $y+\lambda\cdot X$ we denote the set $\{y+\lambda x\mid x\in X\}$.
\end{nota}
\begin{nota}
    Given a finite sequence of numbers $w_1,\dots ,w_n\in\Rbb$ and $i\in\{1,\dots ,n\}$, we denote by $w_1\dots \hat{w_i}\dots w_n$ the product $w_1\dots \hat{w_i}\dots w_n:=w_1\dots w_{i-1}w_{i+1}\dots w_n$.
\end{nota}

\section{Convex geometry}\label{convexgeo}
The proof of Theorem \ref{mainthmconvex} primarily involves studying the geometry of convex sets. We first recall a well-known result in convex geometry.

\begin{lem}\label{johnellipsoid}
    Let $\Omega\subseteq \Rbb^n$ be an open, bounded, convex set. There exists a unique open ellipsoid $J\subseteq\Omega$, known as the \textit{John ellipsoid} of $\Omega$, that has the largest volume among all open ellipsoids contained in $\Omega$. Moreover if $J$ is centered at the origin, then $\Omega\subseteq n\cdot J$. 
\end{lem}
\begin{proof}
    See Theorem 10.12.2 \cite{convexbook}.
\end{proof}

Since the Laplacian commutes with pullbacks by isometries, we may make use of a convenient choice of coordinates to facilitate our exposition. 
\begin{ass}\label{orientation}
    Given a bounded convex domain $\Omega\subseteq\Rbb^n$, assume without loss of generality that its John ellipsoid $J$ is centered at the origin with axes contained in the coordinate axes. Moreover, assume that the axes have decreasing lengths $$a_1\geq a_2\geq \dots \geq a_n>0,$$ where $a_j$ is the length of the axis of $J$ contained in the $j$th coordinate axis. 
\end{ass}

\begin{prop}\label{existrectangles}
    Given a bounded convex domain $\Omega\subseteq\Rbb^n$ satisfying Assumption \ref{orientation}, there exist rectangular prisms $R,Q$ satisfying
    $$R:=\Big[-\frac{a_1}{\sqrt{n}},\frac{a_1}{\sqrt{n}}\Big]\times\dots \times\Big[-\frac{a_n}{\sqrt{n}},\frac{a_n}{\sqrt{n}}\Big]\subseteq \Omega\subseteq \Big[-a_1n,a_1n\Big]\times\dots \times\Big[-a_nn,a_nn\Big]=:Q.$$
\end{prop}
\begin{proof}
   Let $\tilde{R}:=\Big[-\frac{1}{\sqrt{n}},\frac{1}{\sqrt{n}}\Big]^n\subseteq \Rbb^n$ and $\tilde{Q}:=[-1,1]^n\subseteq\Rbb^n$. Then $R\subseteq B^n\subseteq Q$, where $B^n$ is the unit ball in $\Rbb^n$. Since $J$ is the image of $B^n$ under the diagonal linear map $A$ that scales the $j$th standard unit vector by $a_j$, $$R=A(\tilde{R})\subseteq A(B^n)=J\subseteq \Omega\;\;\text{and}\;\;\Omega\subseteq n\cdot J=n\cdot A(B^n)\subseteq n\cdot A(\tilde{Q})=Q.$$
\end{proof}

Using these rectangular prisms, we can immediately give a proof that the largest cross section of a bounded convex domain gives an upper bound on the surface area. This proves that Corollary \ref{crosssectionest} follows from Theorem \ref{mainthmconvex}. Let $A_{n-1}(\Omega)$ be as in Section \ref{intro}. 

\begin{lem}\label{crosssectionperimeter}
    For any convex body $\Omega\subseteq \Rbb^n$, $|\partial\Omega|\leq  \frac{(2n)^n}{\omega_{n-1}}A_{n-1}(\Omega)$. In particular, Theorem \ref{mainthmconvex} implies Corollary \ref{crosssectionest}.
\end{lem}
\begin{proof}
    Since $J\subseteq \Omega$, we have 
    \begin{align*}
        A_{n-1}(\Omega)\geq A_{n-1}(J)\geq |J\cap \{x_n=0\}|=\omega_{n-1}a_1\dots a_{n-1}.
    \end{align*}
    Let $Q$ be as in Proposition \ref{existrectangles}. By Proposition 3.1 \cite{lecturesconvexgeo}, the inclusion $\Omega\subseteq Q$ implies that $|\partial\Omega|\leq |\partial Q|$. Combining these inequalities and using that $a_n\leq a_i$ for all $i$ gives 
    \begin{align*}
        |\partial\Omega|&\leq |\partial Q|\\
        &=2\sum_{i=1}^n(2a_1n)\dots \widehat{(2a_in)}\dots (2a_nn)\\
        &=2^nn^{n-1}\sum_{i=1}^na_1\dots \hat{a_i}\dots a_n\\
        &\leq (2n)^na_1\dots a_{n-1}\\
        &\leq \frac{(2n)^n}{\omega_{n-1}}A_{n-1}(\Omega).
    \end{align*}
\end{proof}

Two more technical lemmas are needed to prove Theorem \ref{mainthmconvex}. Suppose that $\Omega\subseteq \Rbb^n$ satisfies Assumption \ref{orientation}, and let $R\subseteq \Omega$ be the rectangular prism associated to $\Omega$ as in Proposition \ref{existrectangles}.

\begin{lem}\label{rectanglequantities}
    Let $\rho=\sqrt{\frac{1}{a_1^2}+\frac{1}{a_2^2}+\dots +\frac{1}{a_n^2}}$. Then $$2^{-n}\cdot n^{-n/2}\cdot\frac{|\partial R|^n}{|R|^{n-1}}\leq a_1 a_2\dots a_n\rho^n\leq 2^{-n}\cdot\frac{|\partial R|^n}{|R|^{n-1}}.$$
\end{lem}
\begin{proof}
    This is just a direct computation: 
    \begin{align*}
        \frac{|\partial R|^n}{|R|^{n-1}}&=\frac{2^n\Big(\displaystyle\sum_{i=1}^n a_1 \dots \hat{a_i}\dots a_n\Big)^n}{(a_1\dots a_n)^{n-1}},
    \end{align*}
    and 
    \begin{align*}
        a_1\dots a_n\rho^n &= a_1\dots a_n\Bigg(\frac{\displaystyle\sum_{i=1}^n(a_1\dots \hat{a_i}\dots a_n)^2}{(a_1\dots a_n)^2}\Bigg)^{n/2}\\
        &\leq \frac{1}{(a_1\dots a_n)^{n-1}}\Big(\sum_{i=1}^na_1\dots \hat{a_i}\dots a_n\Big)^n\\
        &= 2^{-n}\cdot\frac{|\partial R|^n}{|R|^{n-1}},
    \end{align*}
    where we have applied the inequality $\sqrt{\sum x_i^2}\leq \sum |x_i|$. Applying the reverse inequality $\sqrt{\sum x_i^2}\geq \frac{1}{\sqrt{n}}\sum |x_i|$ gives the other estimate: 
    \begin{align*}
        a_1\dots a_n\rho^n &= a_1\dots a_n\Bigg(\frac{\displaystyle\sum_{i=1}^n(a_1\dots \hat{a_i}\dots a_n)^2}{(a_1\dots a_n)^2}\Bigg)^{n/2}\\
        &\geq \frac{1}{(a_1\dots a_n)^{n-1}}\Big(\frac{1}{\sqrt{n}}\sum_{i=1}^n a_1\dots \hat{a_i}\dots a_n\Big)^n\\
        &\geq 2^{-n}\cdot n^{-n/2}\cdot \frac{|\partial R|^n}{|R|^{n-1}}.
    \end{align*}
\end{proof}

\begin{lem}\label{isoperimetricbounds}
    The isoperimetric ratios of $\Omega$ and $R$ are related by 
    $$n^{-3n(n-1)/2}\cdot \frac{|\partial R|^{n}}{|R|^{n-1}}\leq \frac{|\partial\Omega|^n}{|\Omega|^{n-1}}\leq n^{3n(n-1)/2}\cdot \frac{|\partial R|^n}{|R|^{n-1}}.$$
\end{lem}
\begin{proof}
    Recall that, by Proposition 3.1 \cite{lecturesconvexgeo}, the inclusions $R\subseteq \Omega\subseteq Q$ imply that $|R|\leq |\Omega|\leq |Q|$ and $|\partial R|\leq |\partial \Omega|\leq |\partial Q|$. Note also that $Q=n^{3/2}\cdot R$. Thus, we have $$\frac{|\partial\Omega|^n}{|\Omega|^{n-1}}\leq \frac{|\partial Q|^n}{|R|^{n-1}}=n^{3n(n-1)/2}\cdot \frac{|\partial R|^n}{|R|^{n-1}},$$ and the other inequality is deduced similarly. 
\end{proof}

\section{Proof of Theorem \ref{mainthmconvex}}\label{mainproofsection}

\begin{proof}[Proof of Theorem \ref{mainthmconvex}]
    Let $R\subseteq \Omega\subseteq Q$ be as in Assumption \ref{orientation} and Proposition \ref{existrectangles}. Using Lemma \ref{isoperimetricbounds}, we need only estimate $N(\Omega)$ with respect to the isoperimetric ratio for $R$. 
    By Theorem \ref{funanoestimate}, there exists $E>0$ depending only on $n$ such that for each $j\geq 1$,
    \begin{equation}\label{neuineq}
        \frac{1}{E\cdot n^2 }\cdot \mu_j(Q)\leq \mu_j(\Omega)\leq E\cdot n^2\cdot \mu_j(R),
    \end{equation}
    and we may take $E=92^2$.
    Let $\rho$ be as in Lemma \ref{rectanglequantities}. By Dirichlet domain monotonicity, we have 
    \begin{equation}\label{dirineq}
        \frac{\pi^2}{4n^2}\cdot \rho^2=\lambda_1(Q)\leq \lambda_1(\Omega)\leq \lambda_1(R)=\frac{n\pi^2}{4}\cdot \rho^2.
    \end{equation}
    Therefore,
    \begin{align*}
        N(\Omega)&\geq \#\Big\{j\mid E\cdot n^2 \cdot \mu_j(R)\leq \frac{\pi^2}{4n^2}\cdot \rho^2\Big\}\\
        &= \#\Big\{(j_1,\dots ,j_n)\in \Zbb_{\geq 0}^n\mid E\cdot n^2\cdot \frac{n\pi^2}{4}\Big(\frac{j_1^2}{a_1^2}+\dots +\frac{j_n^2}{a_n^2}\Big)\leq \frac{\pi^2}{4n^2}\cdot \rho^2\Big\}  \\
        &= \#\Big\{(j_1,\dots ,j_n)\in \Zbb_{\geq 0}^n\mid \frac{j_1^2}{a_1^2}+\dots +\frac{j_n^2}{a_n^2}\leq \frac{1}{E\cdot n^5}\cdot \rho^2\Big\}
    \end{align*}
    In other words, $N(\Omega)$ is bounded below by the number of non-negative integer points in the ellipsoid $\mathcal{E}$ centered at the origin and with axes of lengths $\frac{a_i\rho}{\sqrt{E\cdot n^5}}$. Let $T$ denote the union of closed unit cuboids in $\Rbb^n$ whose faces are parallel to the coordinate hyperplanes and whose vertices with smallest coordinates are located at the non-negative integer points of $\mathcal{E}$. Then the number of these points equals the volume $|T|$ of $T$. For any $(x_1,\dots ,x_n)\in \mathcal{E}$, the point $(\lfloor x_1\rfloor,\dots ,\lfloor x_n\rfloor)$ is a non-negative integer point of $\mathcal{E}$, so $(x_1,\dots , x_n)\in T$. Hence, using Lemma \ref{rectanglequantities},
    \begin{align*}
        N(\Omega)\geq |T|\geq |\mathcal{E}|= \frac{\omega_n}{(E\cdot n^5)^{n/2}}a_1\dots a_n\rho^n\geq \frac{\omega_n}{(E\cdot n^5)^{n/2}}\cdot 2^{-n}\cdot n^{-n/2}\cdot \frac{|\partial R|^n}{|R|^{n-1}}.
    \end{align*}
    Then Lemma \ref{isoperimetricbounds} gives the final lower bound on $N(\Omega)$.\\
    \indent For the upper bound on $N(\Omega)$, by (\ref{neuineq}) and (\ref{dirineq}), we have 
    \begin{align*}
        N(\Omega)&\leq \#\Big\{j\mid \frac{1}{E\cdot n^2}\cdot \mu_j(Q)\leq \frac{n\pi^2}{4}\cdot \rho^2\Big\}\\
        &=\#\Big\{(j_1,\dots ,j_n)\in\Zbb_{\geq 0}^n\mid \frac{1}{E\cdot n^2}\cdot \frac{\pi^2}{4n^2}\cdot\Big(\frac{j_1^2}{a_1^2}+\dots +\frac{j_n^2}{a_n^2}\Big)\leq \frac{n\pi^2}{4}\cdot \rho^2\Big\}\\
        &=\#\Big\{(j_1,\dots ,j_n)\in\Zbb_{\geq 0}^n\mid \frac{j_1^2}{a_1^2}+\dots +\frac{j_n^2}{a_n^2}\leq E\cdot n^5\cdot \rho^2\Big\}.
    \end{align*}
    Therefore, $N(\Omega)$ is bounded above by the number of non-negative integer points in the ellipsoid $\mathcal{F}$ with axes of lengths $\sqrt{E\cdot n^5}\cdot a_i\rho$. Let $S$ equal the union of closed unit cuboids with lowest vertex at the non-negative integer points of $\mathcal{F}$. Then $S$ is contained in the rectangular prism $$K:=\Big[0,\big\lceil \sqrt{E\cdot n^5}\cdot a_1\rho\big\rceil\Big]\times\dots \times\Big[0,\big\lceil \sqrt{E\cdot n^5}\cdot a_n\rho\big\rceil\Big].$$ For each $i$, $\sqrt{E\cdot n^5}\cdot a_i\rho>1$, so $\lceil \sqrt{E\cdot n^5}\cdot a_i\rho\rceil \leq 2\sqrt{E\cdot n^5}\cdot a_i\rho$. Thus, using Lemma \ref{rectanglequantities} again,
    $$N(\Omega)\leq |S|\leq |K|\leq 2^n\cdot E^{n/2}\cdot n^{5n/2}\cdot a_1\dots a_n\rho^n\leq 2^n\cdot E^{n/2}\cdot n^{5n/2}\cdot 2^{-n}\cdot \frac{|\partial R|^n}{|R|^{n-1}},$$ and the upper bound on $N(\Omega)$ follows from Lemma \ref{isoperimetricbounds}.
\end{proof}

\section{Non-convex counterexamples}\label{counterexsection}

We present in this section a series of counterexamples showing that neither Theorem \ref{mainthmconvex} nor Corollary \ref{crosssectionest} can be extended to non-convex domains. For simplicity of exposition, the counterexamples we give are all in dimension $2$ but can be extended without much extra effort to higher dimensions.

\begin{thm}\label{counterexs}
    For each $M>0$, there exists a bounded Lipschitz domain $\Omega\subseteq \Rbb^2$ such that $$\frac{|\partial\Omega|}{\sqrt{|\Omega|}}\geq M\;\;\text{and}\;\;N(\Omega)=3.$$
\end{thm}

The construction used to prove Theorem \ref{counterexs} is fairly simple: we consider a sequence of domains that converge in Hausdorff distance to the unit disk while their isoperimetric ratios tend to infinity. Using standard techniques in eigenvalue perturbation theory, we show that the eigenvalues of the sequence of domains tend to those of the disk. \\
Let $B$ be the unit ball in $\Rbb^2$ centered at the origin. Let $g(x)=\sqrt{1-x^2}$, and for each $n\geq 1$, let $f_k(x)=\frac{1}{k}\Big(1+\cos(k^2\pi x)\Big)$. Define $$\Omega_k:=\{(x,y)\in\Rbb^2\mid -1<x<1, -g(x)<y<g(x)+f_k(x)\}.$$ See Figure \ref{spikydisk}. 
\begin{figure}
    \centering
    \includegraphics[width=0.35\linewidth]{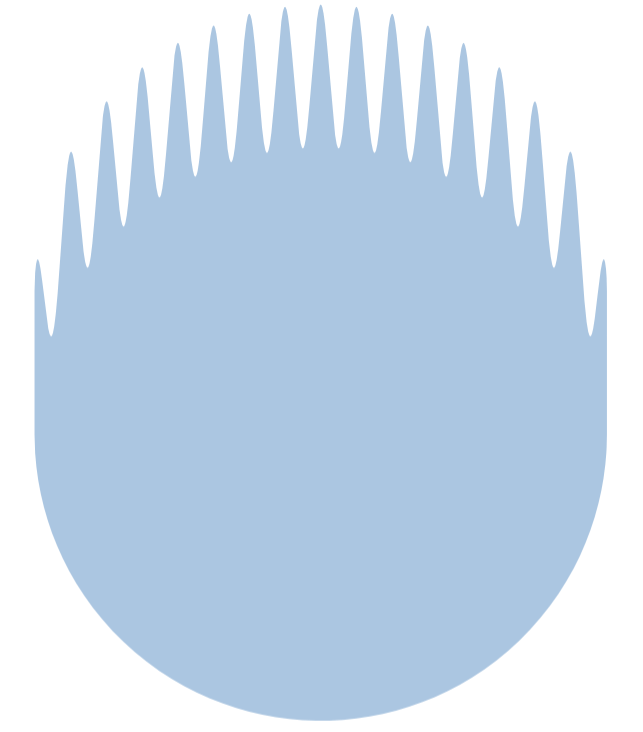}
    \caption{The domain $\Omega_k$ constructed in the proof of Theorem \ref{counterexs} with $k=4$. Illustrated in Desmos (2025).}
    \label{spikydisk}
\end{figure}
\begin{lem}\label{neumannconverge}
    For each $j$, we have $$\lim_{k\to\infty}\mu_j(\Omega_k)=\mu_j(B).$$
\end{lem}
\begin{proof}
    For each $k$, define $P_k=\Omega_k\setminus B$. By Proposition 2.3 of Arrieta-Carvalho \cite{arrieta}, it suffices to prove that the following sequence of eigenvalue problems has first eigenvalue tending to $\infty$ as $k\to\infty$:
    \begin{equation}\label{mixedproblem}
        \begin{cases}
            -\Delta u=\lambda u\;&\text{in}\;\;P_k\\
            u=0\;&\text{in}\;\; \partial P_k\cap \partial B\\
            \partial_{\nu}u=0&\text{in} \;\;\partial P_k\setminus\partial B.
        \end{cases}
    \end{equation}
    Let $\nu_1(P_k)$ denote the first eigenvalue of this problem. Let $H^1_0(P_k)$ denote the Sobolev space of functions in $L^2(P_k)$ whose gradients are in $L^2(P_k)$ and whose traces vanish on $\partial P_k\cap \partial B$. The variational characterization of the eigenvalue problem (\ref{mixedproblem}) is $$\nu_1(P_k)=\min_{u\in H^1_0(P_k)\setminus\{0\}}\frac{\int_{P_k}|\nabla u|^2}{\int_{P_k}|u|^2}.$$ Let $u\in C^1(P_k)\cap H^1_0(P_k)$. Then for any $(x,y)\in P_k$, $$|u(x,y)|^2=\Big|\int_{g(x)}^yu_y(x,s)ds\Big|^2\leq f_k(x)\int_{g(x)}^{g(x)+f_k(x)}|\nabla u(x,s)|^2ds,$$ implying that $$\int_{P_k}|u|^2\leq \Big(\sup_{-1<x<1}f_k(x)^2\Big)\int_{P_k}|\nabla u|^2.$$ By the density of $C^1(P_k)\cap H^1_0(P_k)$ in $H^1_0(P_k)$, we get the lower bound $$\nu_1(P_k)\geq \Big(\sup_{-1<x<1}f_k(x)^2\Big)^{-1}=\frac{k}{2}\to\infty\;\;\text{as}\;\;k\to\infty.$$
\end{proof}

That $N(\Omega_k)\to 3$ as $k\to\infty$ quickly follows:

\begin{proof}[Proof of Theorem \ref{counterexs}]
    We will show that for the $\Omega_k$ constructed above, $N(\Omega_k)=3$ for $k$ sufficiently large, and then we will show that the isoperimetric ratio of $\Omega_k$ tends to infinity as $k$ does. By explicit computation, one can show that $N(B)=3$. By Theorem 1.5 of Rauch-Taylor \cite{rauchtaylor}, the first Dirichlet eigenvalue of $\Omega_k$ tends to the first Dirichlet eigenvalue of $B$ as $k\to\infty$. Combining this with Lemma \ref{neumannconverge}, we have $N(\Omega_k)\to 3$ as $k\to\infty$. \\
    \indent Now note that the area of $\Omega_k$ is bounded above uniformly in $k$, so we need only show that the perimeter of $\Omega_k$ goes to $\infty$ as $k\to\infty$. Since $f_k$ has more than $k^2$ periods contained in the union of intervals $(-1,\epsilon)\cup(\epsilon,1)$ for $\epsilon>0$ small enough and on each of these intervals $g$ is monotonic, the arc length of $g(x)+f_k(x)$ (and hence the perimeter length of $\Omega_k$) is at least $k^2\cdot\frac{2}{k}=2k$.
\end{proof}

\indent In light of Corollary \ref{crosssectionest}, it is natural to wonder, at least in dimension $n=2$, whether there is a lower bound on $N(\Omega)$ in terms of the diameter of $\Omega$.\footnote{We define the \textit{diameter} of $\Omega$ to be the supremum of (Euclidean) distances between pairs of points in $\Omega$.} The next result shows that such a bound also cannot hold. We remark that the sequence of counterexamples we find also gives counterexamples to the isoperimetric lower bound. Let $D(\Omega)$ denote the diameter of $\Omega$.
\begin{thm}\label{diamcounterex}
    There exists a sequence of domains $\Omega_k\subseteq \Rbb^2$ such that $\frac{D(\Omega_k)}{\sqrt{|\Omega_k|}}\to\infty$ while $N(\Omega_k)$ is uniformly bounded above in $k$.
\end{thm}
\noindent In contrast to Theorem \ref{counterexs}, however, we do not know the best uniform upper bound on $N(\Omega_k)$. An example of the domains we construct in the proof is shown in Figure \ref{gaussfig}.
\begin{figure}
    \centering
    \includegraphics[width=0.55\linewidth]{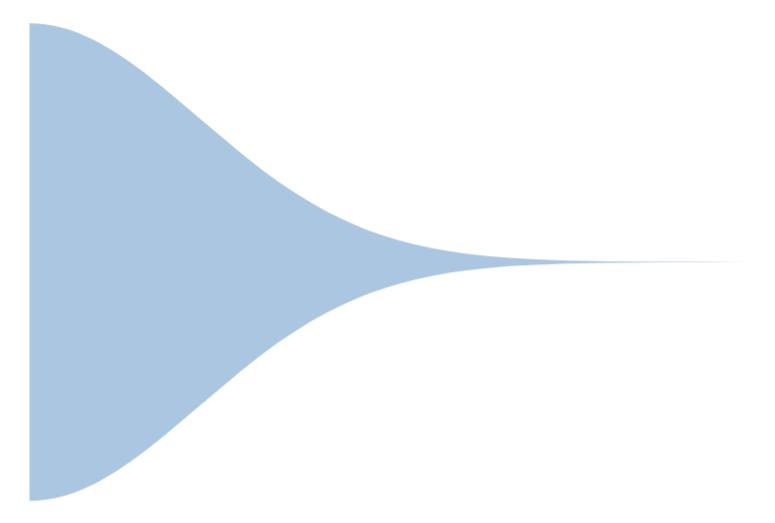}
    \caption{The domain $\Omega_k$ constructed in the proof of Theorem \ref{diamcounterex} with $k=3$. Illustrated in Desmos (2025).}
    \label{gaussfig}
\end{figure}
\begin{proof}[Proof of Theorem \ref{diamcounterex}]
    Define $$\Omega:=\{(x,y)\in \Rbb^2\mid x>0, -e^{-x^2}<y<e^{-x^2}\},$$ and for $k\geq 1$ an integer, define $$\Omega_k:=\{(x,y)\in \Rbb^2\mid 0<x<k, -e^{-x^2}<y<e^{-x^2}\}.$$ Since $\Omega_1\subseteq\Omega_k$ for all $k$, Dirichlet domain monotonicity implies that $\lambda_1(\Omega_k)$ is bounded above by $\lambda_1(\Omega_1)$. By Example 6.3.3 of Evans-Harris \cite{evans}, despite $\Omega$ being unbounded, the Neumann Laplace spectrum of $\Omega$ is discrete and consists only of eigenvalues. In particular, the embedding $H^1(\Omega)\hookrightarrow L^2(\Omega)$ is compact. Moreover, $D(\Omega_k)/\sqrt{|\Omega_k|}\to\infty$ as $k\to \infty$. By Lemma \ref{evalsconverge} below, $\mu_j(\Omega_{k})\to \mu_j(\Omega)$ as $k\to\infty$. By the discreteness of the spectrum of $\Omega$, this completes the proof. 
\end{proof}
A standard technical lemma is needed to prove Lemma \ref{evalsconverge}. 
\begin{lem}\label{subseqconverge}
    For each $k$, suppose that $u_k\in H^1(\Omega_k)$ such that $\|u\|_{L^2(\Omega_k)}=1$ and $\|\nabla u\|_{L^2(\Omega_k)}\leq C$ for some $C>0$. Then there exists $u\in H^1(\Omega)\setminus\{0\}$ and a subsequence $\{u_{k_{\ell}}\}$ such that for any $v\in H^1(\Omega)$, 
    $$\lim_{\ell\to\infty}\int_{\Omega_{k_{\ell}}}\nabla u_{k_{\ell}}\cdot \nabla v=\int_{\Omega} \nabla u\cdot \nabla v.$$
\end{lem}
\begin{proof}
    Let $\chi:\Rbb^2\to\Rbb$ be a smooth function such that $\chi\equiv 1$ on the disk of radius $1$ and $\chi\equiv 0$ outside the disk of radius $2$. For each $k$, define extension operators $E_k:H^1(\Omega_k)\to H^1(\Omega)$ such that for any $f\in C^1(\overline{\Omega_k})$, $$(E_kf)(x,y)=\begin{cases}
        f(x,y)\;&\text{if}\;\;(x,y)\in\Omega_k\\
        \chi(x-k,y)\Big(4f\big(\frac32k-\frac12x,y\big)-3f(2k-x,y)\Big)\;&\text{if}\;\;(x,y)\in \Omega\setminus \Omega_k.
    \end{cases}$$
    Then the operator norms $\|E_k\|$ are bounded above uniformly in $k$. Set $h_k:=E_ku_k\in H^1(\Omega)$. Since the $u_k$ are uniformly bounded in $H^1$ norm, so are the $h_k$. Say $\|\nabla h_k\|_{L^2(\Omega_k)}\leq \tilde{C}$ for all $k$. By Example 6.3.3 of \cite{evans}, the natural embedding $H^1(\Omega)\hookrightarrow L^2(\Omega)$ is compact. Thus, there exists $u\in H^1(\Omega)$ and a subsequence $h_{k_{\ell}}$ converging to $u$ in $L^2(\Omega)$ and converging weakly to $u$ in $H^1(\Omega)$. Let $v\in H^1(\Omega)$, and let $\epsilon>0$. Choose $L$ sufficiently large such that for any $\ell\geq L$, $$\|v\|_{H^1(\Omega\setminus\Omega_{k_{\ell}})}<\frac{\epsilon}{2\tilde{C}}\;\;\text{and}\;\;\Big|\int_{\Omega}\nabla h_{k_{\ell}}\cdot \nabla v-\int_{\Omega}\nabla u\cdot\nabla v\Big|<\frac{\epsilon}{2}.$$ Then for $\ell\geq L$,
    \begin{align*}
        \Big|\int_{\Omega_{k_{\ell}}}\nabla u_{k_{\ell}}\cdot \nabla v&-\int_{\Omega}\nabla u\cdot\nabla v\Big|\\&\leq \Big|\int_{\Omega_{k_{\ell}}}\nabla u_{k_{\ell}}\cdot\nabla v-\int_{\Omega}\nabla h_{k_{\ell}}\cdot\nabla v\Big|+\Big|\int_{\Omega}\nabla h_{k_{\ell}}\cdot \nabla v-\int_{\Omega}\nabla u\cdot\nabla v\Big|\\
        &< \Big|\int_{\Omega\setminus\Omega_{k_{\ell}}}\nabla h_{k_{\ell}}\cdot\nabla v\Big|+\frac{\epsilon}{2}\\
        &\leq \|\nabla h_{k_{\ell}}\|_{L^2(\Omega)}\cdot\frac{\epsilon}{2\tilde{C}}+\frac{\epsilon}{2}\\
        &\leq \epsilon.
    \end{align*}
\end{proof}
\begin{lem}\label{evalsconverge}
    For all $j\geq 1$, $\mu_j(\Omega_k)\to\mu_j(\Omega)$ as $k\to\infty$. 
\end{lem}
\begin{proof}
    Let $\{\phi_j\}$ be an orthonormal basis of Neumann eigenfunctions for $\Omega$, and for each $k$, let $\{\phi_{k,j}\}$ be an orthonormal basis of Neumann eigenfunctions for $\Omega_k$. Let $\psi_{k,j}$ denote the restriction of $\phi_j$ to $\Omega_k$. Let $V_{k,j}$ denote the vector space of functions spanned by $\{\psi_{k,1},\dots ,\psi_{k,j}\}$. By the min-max characterization of the eigenvalues and the fact that $|\Omega\setminus\Omega_k|\to0$ as $k\to\infty$, we have 
    \begin{equation}\label{limits}
        \mu_j(\Omega_k)\leq \max_{u\in V_{k,j}\setminus\{0\}}\frac{\int_{\Omega_k}|\nabla u|^2}{\int_{\Omega_k}|u|^2}=\mu_j(\Omega)+o(1)\;\;\text{as}\;\;k\to\infty.
    \end{equation}
    Thus, each of the sequences $\{\mu_j(\Omega_k)\}_{k=1}^{\infty}$ is bounded. By a diagonalization argument, we may extract a subsequence such that $\mu_j(\Omega_{k_{\ell}})\to\kappa_j\leq \mu_j(\Omega)$ for some $\kappa_j>0$ and for each $j$. By (\ref{limits}) and Lemma \ref{subseqconverge}, we may apply another diagonalization argument to extract a further sequence, which we still label $\{\Omega_{k_{\ell}}\}$, such that for each $j$, there exists $\eta_j\in H^1(\Omega)\setminus\{0\}$ such that for any $v\in H^1(\Omega)$, 
    $$\lim_{\ell\to\infty}\int_{\Omega_{k_{\ell}}}\nabla \phi_{k_{\ell},j}\cdot \nabla v=\int_{\Omega}\nabla \eta_j\cdot\nabla v.$$ Since the $\phi_{k_{\ell},j}$ are eigenfunctions, 
    \begin{align*}
        \int_{\Omega}\nabla \eta_{j}\cdot\nabla v=\lim_{\ell\to\infty}\int_{\Omega_{k_{\ell}}}\nabla \phi_{k_{\ell},j}\cdot\nabla v=\lim_{\ell\to\infty}\mu_j(\Omega_{k_{\ell}})\int_{\Omega_{k_{\ell}}}\phi_{k_{\ell},j}v=\kappa_j\int_{\Omega}\eta_jv.
    \end{align*}
    Therefore, each $\eta_j$ is a Neumann eigenfunction of $\Omega$ with eigenvalue $\kappa_j$. Since $\kappa_j\leq \mu_j(\Omega)$ for each $j$, we in fact have $\kappa_j=\mu_j(\Omega)$ for each $j$. Since we may reapply this argument to any subsequence of $\{\Omega_k\}$, it must be that the eigenvalues of the entire sequence converge to those of $\Omega$. 
\end{proof}

\indent We also construct a sequence of examples showing that the upper bounds in Theorem \ref{mainthmconvex} and Corollary \ref{crosssectionest} also do not hold in the non-convex setting, disproving the upper bound conjectured in \cite{coxetal}. The counterexamples used here are similar to the ``rooms and passages" domains that were studied by many authors (e.g. \cite{evans}) as examples of bounded domains whose essential Neumann spectra are nonempty. 
\begin{thm}\label{uppercounterex}
    There exists a sequence of domains $\Omega_k\subseteq \Rbb^2$ such that both $|\partial\Omega_k|^2/|\Omega|$ and $D(\Omega_k)^2/|\Omega|$ are uniformly bounded above in $k$ while $N(\Omega_k)\to\infty$ as $k\to\infty$.
\end{thm}
\begin{proof}
    We begin by defining the domains $\Omega_k$. For $0\leq j\leq k-1$, define a domain $T_j$ to be the union of the rectangles $[0,3\cdot 2^{-j}]\times [0,\frac{1}{2}\cdot 2^{-3j}]$ and $[2^{-j},2\cdot 2^{-j}]\times [0,2^{-j}]$. For $1\leq j\leq k-1$, translate $T_j$ horizontally such that the left-most edge of $T_j$ is contained in the right-most edge of $T_{j-1}$. Let $\Omega_k$ denote the union of the translated $T_j$'s. See Figure \ref{roomsandpassages}. 
    \begin{figure}
        \centering
        \includegraphics[width=1.0\linewidth]{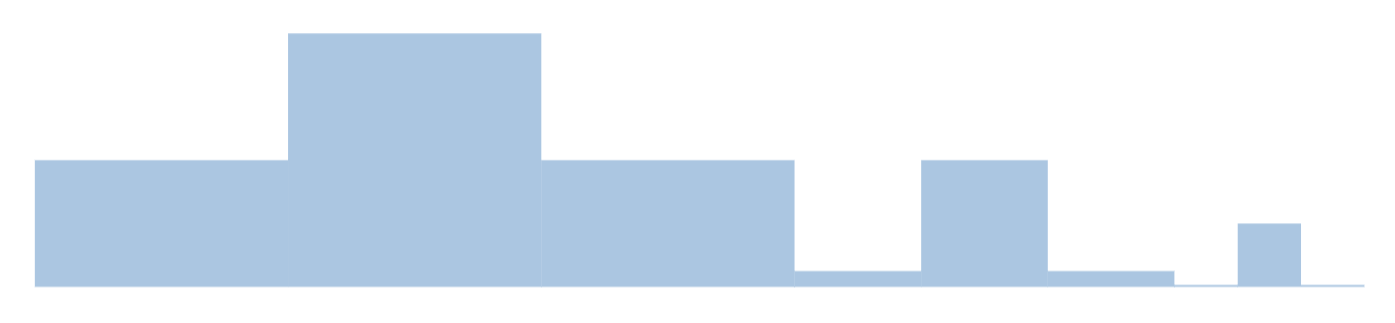}
        \caption{The domain $\Omega_k$ constructed in the proof of Theorem \ref{uppercounterex} with $k=3$. Illustrated in Desmos (2025).}
        \label{roomsandpassages}
    \end{figure}
    A simple computation shows that $|\partial\Omega_k|^2/|\Omega_k|$ and  $D(\Omega_k)^2/|\Omega_k|$ are both bounded above uniformly in $k$.\\
    \indent Since $\Omega_k\subseteq [0,6]\times [0,1]$ for all $k$, Dirichlet domain monotonicity gives a uniform lower bound on the first Dirichlet eigenvalue of $\Omega_k$: $$\lambda_1(\Omega_k)\geq \lambda_1([0,6]\times [0,1])\geq \pi^2.$$
    \indent We claim that $\mu_k(\Omega_k)\leq 1$, from which it follows that $N(\Omega_k)\geq k\to\infty$ as $k\to\infty$. Indeed, let $\nu_1(T_j)$ denote the first mixed eigenvalue of $T_j$ with Dirichlet conditions on the left-most and right-most edges and Neumann conditions elsewhere. By domain monotonicity, $$\mu_k(\Omega_k)\leq \max_{j}\big\{\nu_1(T_j)\big\}.$$ Using a piecewise linear test function
    $$\phi_j(x,y)=\begin{cases}
        2^jx\;\;&\text{if}\;\; 0\leq x\leq 2^{-j}\\
        1\;\;&\text{if}\;\; 2^{-j}\leq x\leq 2\cdot 2^{-j}\\
        -2^jx+3\;\;&\text{if}\;\; 2\cdot 2^{-j}\leq x\leq 3\cdot 2^{-j},
    \end{cases}$$
    we have $$\mu_k(\Omega)\leq \nu_1(T_j)\leq \frac{\int_{T_j}|\nabla \phi|^2}{\int_{T_j}|\phi|^2}\leq 1.$$
\end{proof}

\section{Polygonal domains}
This section is devoted to proving Theorem \ref{polygonthm}. We begin with the lower bound on $N(P)$ for $P\in \mathcal{P}_m$, for which the proof is fairly straightforward.

\begin{proof}[Proof of Theorem \ref{polygonthm}, Part 1]
    Recall the Faber-Krahn inequality, which states that $$\lambda_1(P)\geq \frac{C}{|P|}.$$ where $C:=\pi j_0^2$ and $j_0$ is the first positive zero of the Bessel function of order zero. Let $e$ be a longest edge of $P$, and let $L:=|e|$. Note that since $P$ has $m$ edges, we have 
    \begin{equation}
        |\partial P|\leq mL.
    \end{equation}
    Applying an isometry if necessary, assume that $e$ is contained in the $x$-axis. Define $h:e\to [0,\infty)$ by $h(x)=\big|\{y\mid (x,y)\in P\}\big|$. Since $P$ is a polygon, $h$ is a piecewise affine function. Say that $h$ is affine on intervals $I_1,\dots , I_M$. Then $M\leq m-1$. Suppose that $I_i$ is the longest of these intervals, and apply an isometry to $P$ if necessary such that $I=[0,2\ell]\times\{0\}$ for some $\ell>0$. Since there are at most $m-1$ intervals and their total length is $L$, it follows that $2\ell\geq \frac{1}{m-1}L\geq \frac{1}{m(m-1)}|\partial P|$. By reflecting about the line $\{x=\ell\}$ if necessary, we may assume that, on $I_i$, $h(x)=ax+b$ with $a\leq 0$ and $b>0$.\\
    \indent Let $$N:=\Bigg\lfloor \sqrt{\frac{C}{2\pi^2}}\cdot \frac{\ell}{\sqrt{|P|}}\Bigg\rfloor.$$ We claim that $\mu_N(P)\leq \lambda_1(P)$. Indeed, on the interval $[0,\ell]$, $h$ satisfies the inequality $\frac{b}{2}\leq h(x)\leq b$.  Let $V_N$ denote the span of the functions, for $1\leq j\leq N$, $$(x,y)\mapsto \begin{cases}
        \sin\Big(\frac{j\pi}{\ell}x\Big)\;\;&\text{if}\;\;x\in[0,\ell]\\
        0&\text{if}\;\;x\notin[0,\ell].
    \end{cases}$$ Then \begin{align*}
        \mu_N(P)&\leq \max_{u\in V_N\setminus\{0\}}\frac{\int_P|\nabla u|^2}{\int_P|u|^2}\\
        &=\max_{u\in V_N\setminus\{0\}}\frac{\int_0^{\ell}|\partial_x u|^2h(x)dx}{\int_0^{\ell}|u|^2h(x)dx}\\
        &\leq 2\max_{u\in V_N\setminus \{0\}}\frac{\int_0^{\ell}|\partial_xu|^2dx}{\int_0^{\ell}|u|^2dx}\\
        &=2\Big(\frac{N\pi}{\ell}\Big)^2\\
        &\leq \frac{C}{|P|}\\
        &\leq \lambda_1(P).
    \end{align*} 
    Since $N(\Omega)\geq 1$ for any Lipschitz domain $\Omega$, we then have $$N(P)\geq \max\{1,N\}\geq \frac12\cdot\sqrt{\frac{C}{2\pi^2}}\cdot \frac{\ell}{\sqrt{|P|}}\geq \frac{j_0}{4\sqrt{2\pi}\cdot m(m-1)}\cdot \frac{|\partial P|}{\sqrt{|P|}}.$$
\end{proof}

\indent Before proving the upper bound in Theorem \ref{polygonthm}, we show that, for $n\geq 4$, the lower bound is sharp up to the constants used. 
\begin{eg}\label{polygoncounterex}
    Let $T\subseteq \Rbb^2$ be the open triangular domain with vertices at $(-1,0), (0,0)$, and $(0,1)$. For $k\geq 5$, let $T_{k}\subseteq \Rbb^2$ be the open triangular domain with vertices at $(0,0),(0,1/k),(k,0)$. Let $$\Omega_k=T\cup T_k\cup\{(0,y)\mid 0<y<1/k\}\subseteq \Rbb^2.$$ See Figure \ref{polygoncounterexfigure}.
    \begin{figure}
        \centering
        \includegraphics[width=1.0\linewidth]{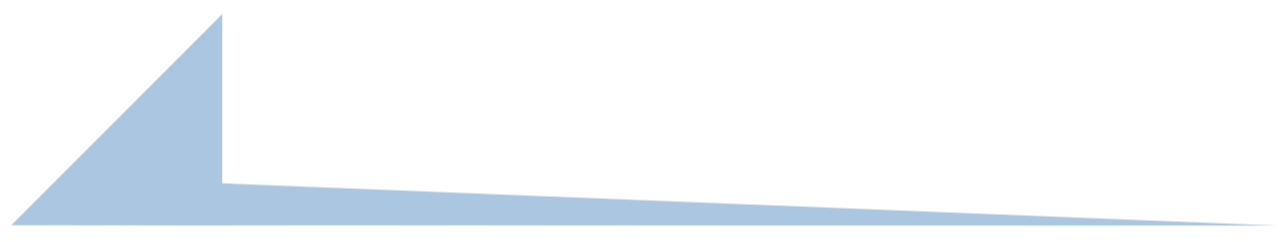}
        \caption{The domain $\Omega_k$ constructed in Example \ref{polygoncounterex} with $k=5$. Illustrated in Desmos (2025).}
        \label{polygoncounterexfigure}
    \end{figure}
    Then there exists $C>0$ independent of $k$ such that $$N(\Omega_k)\leq C\cdot \frac{|\partial\Omega_k|}{\sqrt{|\Omega_k|}}.$$
\end{eg}
\begin{proof}
    First we find an upper bound on $\lambda_1(\Omega_k)$. One can check by direct computation that $\sin(\pi(x+1))\sin(2\pi y)-\sin(2\pi(x+1))\sin(\pi y)$ is a Laplace eigenfunction satisfying Dirichlet boundary conditions on the triangle $T$. Since this function is also positive in $T$, it follows that it is a first Dirichlet eigenfunction, so $\lambda_1(T)=5\pi^2$. By Dirichlet domain monotonicity, therefore, $\lambda_1(\Omega_k)\leq \lambda_1(T)=5\pi^2$.\\
    \indent Let $Q=[0,1/5]\times[0,1/5]\subseteq \Rbb^2$. Let $\{Q_m\}$ be an enumeration of the nonempty intersections $(Q+\frac15(p,q))\cap \Omega_k$ where $(p,q)\in\Zbb^2$. By the definition of $\Omega_k$, at most $15+5k$ of these intersections are nonempty, and each $Q_m$ is convex. By Neumann domain monotonicity, we have $$\mu_{16+5k}(\Omega_k)\geq \min_m\mu_2(Q_m).$$ Since the diameter of each $Q_m$ is at most $\sqrt{2}/5$, a result of Payne and Weinberger \cite{payneweinberger} states that, for each $m$, $$\mu_2(Q_m)\geq \frac{\pi^2}{2/25},$$ so we have $$\mu_{16+5k}(\Omega_k)\geq \lambda_1(\Omega_k),$$ and the desired bound follows. 
\end{proof}
We now prove the upper bound in Theorem \ref{polygonthm}. We begin by describing a process in which an $m$-gon is partitioned into at most $m-2$ disjoint convex polygons. Following this construction, we derive an upper bound on $N(P)$ in terms of the Neumann eigenvalues of the resulting convex polygons and then make estimates on these eigenvalues. 
\begin{proof}[Proof of Theorem \ref{polygonthm}, Part 2]
    Recall that $P$ has at most $m-3$ non-convex vertices (i.e. vertices of interior angle greater than $\pi$). Let $v$ be a non-convex vertex of $P$ (if such a vertex exists). Let $\gamma$ be the longest line segment contained in $P$ that has an endpoint at $v$ and that bisects the interior angle at $v$. Then $P\setminus \gamma$ has at most two connected components and at least one fewer non-convex vertex than $P$. Repeat this process on $P\setminus\gamma $ (note that the first step may have removed two non-convex vertices) if it has a remaining non-convex vertex. Continue repeating this process until no non-convex vertices remain. Since $P$ has at most $m-3$ non-convex vertices, this process terminates in at most $m-3$ steps, and we arrive at a disjoint union of convex polygons that we call $P_M$, which has $M\leq m-2$ connected components. Label these connected components $P_M=:Q_1\cup \dots \cup Q_M$.\\
    \indent Since the $Q_i$ partition $P$, Neumann domain monotonicity implies that $$N(P)\leq \sum_{i=1}^M \#\Big\{j\mid \mu_j(Q_i)\leq \lambda_1(P)\Big\}.$$ To prove the upper bound, we estimate $\lambda_1(P)$ from above and $\mu_j(Q_i)$ from below. To do so, we first define some geometric quantitites related to the $Q_i$.\\
    \indent For each $i\in\{1,\dots ,M\}$, let $\ell_i\subseteq \overline{Q_i}$ be a line segment realizing the diameter of $Q_i$. Let $w_i$ denote the supremum of the lengths of line segments contained in $Q_i$ with an endpoint in $\ell_i$ and that are perpindicular to $\ell_i$. Suppose without loss of generality that $w_1\geq w_i$ for all $i\in\{1,\dots ,M\}$. Let $\gamma_i\subseteq \overline{Q_i}$ denote a line segment realizing $w_i$. \\
    \indent Since $Q_1$ is convex, it contains the convex hull of $\ell_1\cup \gamma_1$, which is a triangle. This triangle contains a rectangle isometric to $[0,d_1/2]\times[0,w_1/2]$. Applying Dirichlet domain monotonicity and the inequality $w_1\leq d_1$, we get $$\lambda_1(P)\leq \lambda_1([0,d_1/2]\times[0,w_1/2])\leq \frac{8\pi^2}{w_1^2}.$$ For each $i$, let $$T_i:=\#\Big\{j\mid \mu_j(Q_i)\leq \frac{8\pi^2}{w_1^2}\Big\},$$ so $N(P)\leq \displaystyle\sum_{i=1}^M T_i$.\\
    \indent Fix $i\in\{1,\dots ,M\}$. By applying an isometry to $Q_i$, assume that $\ell_i$ is contained in the non-negative $x$-axis with an endpoint at the origin. As in Example \ref{polygoncounterex}, cover $Q_i$ with almost disjoint adjacent squares isometric to $[0,w_1/5]^2$. Let $R_1,\dots ,R_{p_i}$ denote the intersections of these squares with $Q_i$. Then each $R_r$ is convex with diameter equal to at most $\sqrt{2}w_1/5$. Neumann domain monotonicity and the Payne-Weinberger inequality \cite{payneweinberger} therefore implies that $$\mu_{p_i+1}(Q_i)\geq \min_{r}\mu_2(R_r)\geq \min_{r}\frac{\pi^2}{\text{diam}(R_r)^2}\geq \frac{25}{2}\cdot \frac{\pi^2}{w_1^2}\geq \frac{8\pi^2}{w_1^2}\geq \lambda_1(P).$$ Therefore $T_i\leq p_i+1$, and it remains to bound the value of $p_i+1$ from above with respect to the isoperimetric ratio of $P$. For this we use that for each $i$, $w_i,d_i\leq \frac{1}{2}|\partial P|$. Since $w_1$ is maximal, we get 
    $$T_i\leq 1+p_i\leq 1+10\cdot \Big\lceil \frac{5d_i}{w_1}\Big\rceil \leq 11\cdot\Big \lceil \frac{5|\partial P|}{2 w_1}\Big\rceil.$$ Since $\frac{5|\partial P|}{2w_1}>1$, its ceiling is at most $\frac{5|\partial P|}{w_1}$, so we get $$T_i\leq 55\cdot \frac{|\partial P|}{w_1}.$$ Summing over $i$ gives $$N(P)\leq \sum_{i=1}^MT_i\leq \sum_{i=1}^M55\cdot \frac{|\partial P|}{w_1}\leq 55(m-2)\cdot \frac{|\partial P|}{w_1}.$$ Now note that, up to isometry, each $Q_i$ is contained in the rectangle $[0,d_i]\times [-w_i,w_i]$, so $$|Q_i|\leq 2w_id_i\leq 2w_1d_i\leq w_1|\partial P|,$$ and summing over $i$ gives $$|P|=\sum_{i=1}^M |Q_i|\leq (m-2)w_1|\partial P|,$$ so we get $$\frac{1}{w_1}\leq (m-2)\cdot\frac{|\partial P|}{|P|}.$$ Plugging this into the last inequality involving $N(P)$ gives the result: 
    $$N(P)\leq 55(m-2)\cdot |\partial P|\cdot (m-2)\cdot \frac{|\partial P|}{|P|}=55(m-2)^2\cdot \frac{|\partial P|^2}{|P|}.$$
\end{proof}

\section{Fiber bundles with shrinking fibers}\label{mfldsection}

In this section we study $N(\Omega)$ for certain families of fiber bundles of smooth compact Riemannian manifolds, which will lead to a proof of Theorem \ref{tubenbhd}. 

\begin{nota}
   We denote the volume of a Riemannian manifold $(\Omega,g_{\Omega})$ by $|\Omega|_{g_{\Omega}}$. If $M\subseteq\Omega$ is a lower-dimensional submanifold, we still denote by $|M|_{g_\Omega}$ the Riemannian volume of $M$ induced by the metric $g_{\Omega}$. When it is clear in context, we may omit the subscript. 
\end{nota}
\begin{nota}
    When a Riemannian manifold $(\Omega,g_{\Omega})$ has empty boundary, we use the same notation for the eigenvalues of the Laplace-Beltrami operator as we use for Neumann eigenvalues, i.e. $\mu_j(\Omega)$. When we wish to emphasize the dependence of the eigenvalues on the metric, we may also write $\mu_j(\Omega,g_{\Omega})$ and similarly for the $\lambda_j$.
\end{nota}
\indent Let $(M,g_M)$ be a compact $m$-dimensional smooth Riemannian manifold with smooth (possibly empty) boundary. Let $(F,g_F)$ be a compact $n$-dimensional smooth Riemannian manifold with smooth nonempty boundary. Let $p:\Omega\to M$ be a smooth fiber bundle with fibers diffeomorphic to $F$. For $\epsilon_0>0$, let $\{\geps\}_{\epsilon\in(0,\epsilon_0)}$ denote a family of Riemannian metrics on $\Omega$. We often denote the pair $(\Omega,\geps)$ by $\Oe$. The general result for fiber bundles is

\begin{thm}\label{asymptoticthm}
    Suppose that there exists $C>0$ and a local trivialization $$\{(U_i\subseteq M,\phi_i:p^{-1}(U_i)\to U_i\times F)\}_{i=1}^K$$ such that for each $i$, the identity map $$\Big(U_i\times F,(\phi_i^{-1})^*(\geps|_{p^{-1}(U_i)})\Big)\to \Big(U_i\times F,g_M|_{U_i}\oplus \epsilon^2g_F\Big)$$ is $(1+C\epsilon)$-bi-Lipschitz. Then, as $\epsilon\to 0^+$, 
    \begin{align*}
        N(\Oe)&=\frac{\omega_m}{(2\pi)^m}\cdot|M|\cdot\Bigg(\sum_{j=1}^{N(F)}\big(\lambda_1(F)-\mu_j(F)\big)^{m/2}\Bigg)\cdot\frac{1}{\epsilon^m}+o\Big(\frac{1}{\epsilon}\Big)\\
        &=\frac{\omega_m}{(2\pi)^m}\cdot\frac{|F|^{n+m-1}}{|\partial F|^{n+m}}\cdot\Bigg(\sum_{j=1}^{N(F)}\big(\lambda_1(F)-\mu_j(F)\big)^{m/2}\Bigg)\cdot\frac{|\partial\Oe|^{n+m}}{|\Oe|^{n+m-1}}+o\Big(\frac{|\partial\Oe|^{n+m}}{|\Oe|^{n+m-1}}\Big)
    \end{align*}
\end{thm}

One can check that for $\{p:\Oe\to M\}_{\epsilon\in(0,\epsilon_0)}$ satisfying the hypotheses of Theorem \ref{asymptoticthm}, $\Oe$ Gromov-Hausdorff converges to $(M,g_M)$ as $\epsilon\to0^+$. Moreover, since $F$ has nonempty boundary, the isoperimetric ratio of $\Oe$ tends to infinity as $\epsilon\to0^+$.\\
\indent We prove Theorem \ref{asymptoticthm} by first establishing the result for trivial bundles with product metrics. Using standard techniques from matrix perturbation theory, we then extend this result to trivial bundles with metrics satisfying the hypotheses of Theorem \ref{asymptoticthm}. We then complete the proof by performing Dirichlet-Neumann bracketing on the trivializing atlas for a general fiber bundle.\\
\indent Before proceeding with the proof, we show that the second two expressions in the conclusion of Theorem \ref{asymptoticthm} are equal to each other. The remainder of the proof is dedicated to proving the first equality. 
\begin{lem}\label{bdysize}
   For each $\epsilon\in (0,\epsilon_0)$, $$\frac{|\partial\Oe|^{n+m}}{|\Oe|^{n+m-1}}=\frac{1}{\epsilon^m}|M|\cdot\frac{|\partial F|^{n+m}}{|F|^{n+m-1}}+O\Big(\frac{1}{\epsilon^{m-1}}\Big).$$
\end{lem}
\begin{proof}
    Let $\{(U_i,\phi_i)\}_{i=1}^K$ be a local trivialization as in the statement of Theorem \ref{asymptoticthm}. For $i\in\{1,\dots ,K\}$, let $V_i\subseteq U_i$ be an open set such that the $V_i$ are mutually disjoint and such that their closures cover $M$. Then each $p^{-1}(V_i)$ is $(1+C\epsilon)$-bi-Lipschitz to $(V_i\times F,g_M|_{V_i}\oplus \epsilon^2g_F)$. This gives, as $\epsilon\to 0^+$ 
    \begin{align*}
        |\Oe|_{\geps}&=\sum_i|p^{-1}(V_i)|_{\geps}\\
        &= \sum_i(1+O(\epsilon))^{n+m}|V_i|_{g_M}\cdot \epsilon^n|F|_{g_F}\\
        &=(1+O(\epsilon))\epsilon^n|M|_{g_M}\cdot|F|_{g_F},
    \end{align*}
    and
    \begin{align*}
        |\partial\Oe|_{\geps}&=\sum_i|\partial(p^{-1}(V_i))\cap \partial \Oe|_{\geps}\\
        &=\sum_i(1+O(\epsilon))^{n+m-1}\Big[|\partial V_i\cap\partial M|_{g_M}\cdot \epsilon^n|F|_{g_F}+|V_i|_{g_M}\cdot\epsilon^{n-1}|\partial F|_{g_F}\Big]\\
        &=(1+O(\epsilon))\epsilon^{n-1}|\partial F|_{g_F}|M|_{g_M}.
    \end{align*}
    Exponentiating and dividing these expressions gives the result. 
\end{proof}

\begin{prop}\label{trivialversion}
    Suppose that $p:\Omega\to M$ is a trivial bundle with a family of metrics equal to $\geps=g_M\oplus \epsilon^2g_F$. Then the conclusion of Theorem \ref{asymptoticthm} holds for the family $\{\Oe\}_{\epsilon\in(0,\epsilon_0)}$.
\end{prop}
\begin{proof}
    Recall that the eigenvalues of a Riemannian product manifolds are the sums of eigenvalues (with multiplicity) of the summands. Thus, $$\lambda_1(\Oe)=\lambda_1(M)+\frac{1}{\epsilon^2}\lambda_1(F),$$ where we take $\lambda_1(M):=0$ if $\partial M=\emptyset$, and for each $j\geq 1$, there exist $\ell,m\geq 1$ such that $$\mu_j(\Oe)=\mu_{\ell}(M)+\frac{1}{\epsilon^2}\mu_k(F).$$ By the definition of $N(F)$, we may suppose for the remainder of the proof that $\epsilon$ is sufficiently small that $\frac{1}{\epsilon^2}\mu_{N(F)+1}(F)>\lambda_1(\Oe)$. This implies that if $\mu_{\ell}(M)+\frac{1}{\epsilon^2}\mu_k(F)\leq \lambda_1(\Oe)$, then $k\leq N(F)$.\\
    \indent Let $\mathcal{N}_M(\lambda)$ denote the number of (Neumann) eigenvalues of $M$ that are less than or equal to $\lambda\geq 0$. Weyl's law (see, for instance, Canzani's notes \cite{canzani}) states that as $\lambda\to+\infty$,
    \begin{equation*}
        \mathcal{N}_M(\lambda)=\frac{\omega_m}{(2\pi)^m}|M|_{g_M}\lambda^{m/2}+o(\lambda^{m/2}).
    \end{equation*}
    By the expression given for the eigenvalues of $\Oe$ and since $\epsilon$ is sufficiently small, therefore, we have 
    \begin{align*}
        N(\Oe)&=\sum_{j=1}^{N(F)}\#\Big\{\ell\mid \mu_{\ell}(M)+\frac{1}{\epsilon^2}\mu_j(F)\leq \lambda_1(M)+\frac{1}{\epsilon^2}\lambda_1(F)\Big\}\\
        &=\sum_{j=1}^{N(F)}\mathcal{N}_M\Big(\lambda_1(M)+\frac{1}{\epsilon^2}\big(\lambda_1(F)-\mu_j(F))\Big)\\
        &=\sum_{j=1}^{N(F)}\frac{\omega_m}{(2\pi)^m}|M|\cdot\big(\lambda_1(F)-\mu_j(F)\big)^{m/2}\cdot\frac{1}{\epsilon^m}+o\Big(\frac{1}{\epsilon^m}\Big)\\
        &=\frac{\omega_m}{(2\pi)^m}\cdot|M|\cdot\Bigg(\sum_{j=1}^{N(F)}\big(\lambda_1(F)-\mu_j(F)\big)^{m/2}\Bigg)\cdot\frac{1}{\epsilon^m}+o\Big(\frac{1}{\epsilon^m}\Big).
    \end{align*}
\end{proof}

\indent We now extend Proposition \ref{trivialversion} to trivial bundles with non-product metrics. Before doing so, we establish some notation.
\begin{nota}
    Suppose that $(\Omega,g_{\Omega})$ is a Riemannian manifold. Let $u:\Omega\to\Cbb$ be a differentiable function. Denote by $\nabla_{g_{\Omega}}u$ the Riemannian gradient of $u$, and let $$|\nabla_{g_\Omega}u|_{g_\Omega}^2:=g_{\Omega}(\nabla_{g_\Omega}u,\nabla_{g_\Omega}u).$$ Let $dV_{g_\Omega}$ denote the Riemannian volume density of $\Omega$. For objects defined by $\Oe$, we often replace the subscript $\Oe$ with simply $\epsilon$; e.g. $\nabla_{\epsilon}:=\nabla_{\Oe}$.
\end{nota}

\begin{prop}\label{trivialwarped}
    Suppose that $p:\Omega\to M$ is a trivial bundle with a family of metrics $\{\geps\}_{\epsilon\in(0,\epsilon_0)}$ satisfying the hypothesis of Theorem \ref{asymptoticthm}. Then the conclusion of Theorem \ref{asymptoticthm} holds for the family $\{\Oe\}_{\epsilon\in(0,\epsilon_0)}$.
\end{prop}
\begin{proof}
    By the hypothesis of Theorem \ref{asymptoticthm} and since the bundle $p:\Omega\to M$ is trivial, there is a $(1+O(\epsilon))$-bi-Lipschitz diffeomorphism $\phi:\Oe\to (M\times F,g_M\oplus \epsilon^2g_F)$. Note that coordinate expressions for the volume form and gradient of a function with respect to a Riemannian metric depend smoothly on the metric. Thus, for a function $u\in H^1(\Oe)$, the Rayleigh quotient of $u$ on $\Oe$ is approximately equal to the Rayleigh quotient of $(\phi^{-1})^*(u)$ on $(M\times F,g_M\oplus\epsilon^2 g_F)$:
    $$\frac{\int_{\Oe}|\nabla_{\epsilon}u|_{\epsilon}^2dV_{\epsilon}}{\int_{\Oe}|u|^2dV_{\epsilon}}=(1+O(\epsilon))\frac{\int_{M\times F}|\nabla_{g_M\oplus \epsilon^2 g_F}(\phi^{-1})^*(u)|_{g_M\oplus\epsilon^2 g_F}^2dV_{g_M\oplus\epsilon^2 g_F}}{\int_{M\times F}|(\phi^{-1})^*(u)|^2dV_{g_M\oplus \epsilon^2 g_F}},$$ where the $O(\epsilon)$ term is independent of $u$, and a similar expression holds for $v\in H^1(M\times F)$.\\
    For $j\geq 1$, let $\mathcal{V}_j$ denote the span of the first $j$ eigenfunctions of the Laplace-Beltrami operator on $(M\times F,g_M\oplus\epsilon^2 g_F)$. Then 
    \begin{align*}
        \mu_j(\Oe)&\leq \max_{u\in \mathcal{V}_j\setminus\{0\}}\frac{\int_{\Oe}|\nabla_{\geps}\phi^*(u)|_{\geps}^2dV_{\epsilon}}{\int_{\Oe}|\phi^*(u)|^2dV_{\epsilon}} \\
        &=(1+O(\epsilon))\max_{u\in \mathcal{V}_j\setminus \{0\}}\frac{\int_{M\times F}|\nabla_{g_M\oplus \epsilon^2 g_F}u|_{g_M\oplus\epsilon^2 g_F}^2dV_{g_M\oplus\epsilon^2 g_F}}{\int_{M\times F}|u|^2dV_{g_M\oplus \epsilon^2 g_F}}\\
        &=(1+O(\epsilon))\mu_j(M\times F,g_M\oplus\epsilon^2 g_F).
    \end{align*}
    By a similar computation, the opposite inequality holds, and we have $$\mu_j(\Oe)=(1+O(\epsilon))\mu_j(M\times F,g_M\oplus \epsilon^2 g_F)$$ where the $O(\epsilon)$ term is independent of $j$.
    The same equality holds for the Dirichlet eigenvalues as well. In particular, there exists $C_1>0$ such that for all $j\geq 1$, 
    $$(1-C_1\epsilon)\mu_j(M\times F,g_M\oplus\epsilon^2 g_F)\leq \mu_j(\Oe)\leq (1+C_1\epsilon)\mu_j(M\times F,g_M\oplus\epsilon^2 g_F)\;\;\text{and}$$
    $$(1-C_1\epsilon)\lambda_1(M\times F,g_M\oplus\epsilon^2 g_F)\leq \lambda_1(\Oe)\leq (1+C_1\epsilon)\lambda_1(M\times F,g_M\oplus\epsilon^2 g_F).$$ Thus, by a slight modification of the computation in the proof of Proposition \ref{trivialversion}, we get that for $\epsilon$ sufficiently small,
    \begin{align*}
        N(\Oe)&\leq \#\Big\{j:(1-C_1\epsilon)\mu_j(M\times F,g_M\oplus\epsilon^2 g_F)\leq (1+C_1\epsilon)\lambda_1(M\times F,g_M\oplus\epsilon^2 g_F)\Big\}\\
        &=\#\Big\{j:\mu_j(M\times F,g_M\oplus\epsilon^2 g_F)\leq (1+O(\epsilon))\lambda_1(M\times F,g_M\oplus\epsilon^2 g_F)\Big\}\\
        &=\sum_{j=1}^{N(F)}\#\Big\{\ell:\mu_{\ell}(M)+\frac{1}{\epsilon^2}\mu_j(F)\leq \frac{1}{\epsilon^2}\lambda_1(F)+O(1)\Big\}\\
        &=\frac{\omega_m}{(2\pi)^m}\cdot |M|\cdot\Big(\sum_{j=1}^{N(F)}(\lambda_1(F)-\mu_j(F))^{m/2}\Big)\cdot\frac{1}{\epsilon^m}+o\Big(\frac{1}{\epsilon^m}\Big).
    \end{align*}
    The other necessary inequality is computed similarly to yield equality. 
\end{proof}

To complete the proof of Theorem \ref{asymptoticthm}, we apply a Dirichlet-Neumann bracketing argument. To do so, we must first establish a modified version of Proposition \ref{trivialwarped} in which we study a mixed Dirichlet-Neumann eigenvalue problem on $\Oe$. If $\Oe$ is a trivial bundle, denote by $N_{\text{mix}}(\Oe)$ the number of eigenvalues of $\Oe$ satisfying Dirichlet conditions on $(\partial M)\times \text{int}(F)$\footnote{Here $\text{int}(F):=F\setminus\partial F$ denotes the interior of $F$.} and Neumann conditions on $\text{int}(M)\times \partial F$ that do not exceed the first Dirichlet eigenvalue of $\Oe$. We refer to the corresponding mixed eigenvalues by $\lambda_j^{\text{mix}}(\Oe)$. The proof of the following lemma is nearly identical to that of Proposition \ref{trivialwarped}, so we omit it. 
\begin{lem}\label{mixedtrivialwarped}
    Suppose that $M$ has nonempty boundary. Let $f(\epsilon)$ be any function that is $O\big(\frac{1}{\epsilon}\big)$. If $\{\Oe\}_{\epsilon\in(0,\epsilon_0)}$ is a family of trivial bundles over $M$ satisfying the hypothesis of Theorem \ref{asymptoticthm}, then $$\#\Big\{j:\lambda_j^{\text{mix}}(\Oe)\leq \lambda_1(\Oe)+f(\epsilon)\Big)\Big\}=N(\Oe)+o\Big(\frac{1}{\epsilon^m}\Big).$$
\end{lem}

We need only one more lemma before proving Theorem \ref{asymptoticthm}, and this lemma may be of general interest for studying $N(\Omega)$ for other domains and manifolds $\Omega$. The result is basically a domain monotonicity statement for the quantity $N(\Omega)$.

\begin{lem}\label{dommon}
    Let $(\Omega,g_{\Omega})$ be a compact Riemannian manifold with nonempty boundary. Suppose that $U_1,\dots ,U_K\subseteq \Omega$ are mutually disjoint open sets with Lipschitz boundary whose closures cover $\Omega$. Then $$N(\Omega)\leq \sum_{i=1}^{K}N(U_i).$$
\end{lem}
\begin{proof}
    Let $0=\nu_1=\dots =\nu_K<\nu_{K+1}\leq \dots $ be an ordering of the union of the Neumann eigenvalues of the sets $U_i$ (with multiplicity). By Neumann domain monotonicity, we have $\nu_j\leq \mu_j(\Omega)$ for all $j\geq 1$. By Dirichlet domain monotonicity, $\lambda_1(\Omega)\leq \lambda_1(U_i)$ for all $i$. Hence,
    \begin{align*}
        N(\Omega)&\leq \#\{j:\nu_j\leq \lambda_1(\Omega)\}\\
        &=\sum_{i=1}^K\#\{j:\mu_j(U_i)\leq \lambda_1(\Omega)\}\\
        &\leq \sum_{i=1}^K\#\{j:\mu_j(U_i)\leq \lambda_1(U_i)\}\\
        &=\sum_{i=1}^KN(U_i).
    \end{align*}
\end{proof}

\begin{proof}[Proof of Theorem \ref{asymptoticthm}]
    Let $\{U_i\}_{i=1}^K$ be the local trivialization as in the statement of the theorem. For each $i$, let $V_i\subseteq U_i$ be an open set with Lipschitz boundary such that the closures of the $V_i$ cover $M$. Then the restriction $p|_{p^{-1}(V_i)}:V_i\to U_i$ is a trivial bundle, so we may apply Proposition \ref{trivialwarped} to these sets. By Lemma \ref{dommon}, we have 
    \begin{align*}
        N(\Oe)&\leq \sum_{i=1}^K N(p^{-1}(V_i))\\
        &=\sum_{i=1}^K \Bigg(\frac{\omega_m}{(2\pi)^m}\cdot|p^{-1}(V_i)|\cdot\Bigg(\sum_{j=1}^{N(F)}\big(\lambda_1(F)-\mu_j(F)\big)^{m/2}\Bigg)\cdot\frac{1}{\epsilon^m}+o\Big(\frac{1}{\epsilon}\Big)\Bigg)\\
        &=\frac{\omega_m}{(2\pi)^m}\cdot|M|\cdot\Bigg(\sum_{j=1}^{N(F)}\big(\lambda_1(F)-\mu_j(F)\big)^{m/2}\Bigg)\cdot\frac{1}{\epsilon^m}+o\Big(\frac{1}{\epsilon}\Big).
    \end{align*}
    To prove the opposite inequality, we first compare the first Dirichlet eigenvalues of $\Oe$ and the $p^{-1}(V_i)$. By the argument given in the proof of Proposition \ref{trivialwarped}, $$\lambda_1(p^{-1}(V_i))=(1+O(\epsilon))\lambda_1(V_i\times F,(g_M|_{V_i})\oplus \epsilon^2 g_F)=\frac{1}{\epsilon^2}\lambda_1(F)+O\Big(\frac{1}{\epsilon}\Big).$$
    By Dirichlet domain monotonicity, it follows that $$\lambda_1(\Oe)\leq \lambda_1(p^{-1}(V_i)).$$\\
    \indent Note that the volume form for the metric $(g_M|_{V_i})\oplus \epsilon^2 g_F$ is given by $\epsilon^ndV_{g_M}dV_{g_F}$ and that the norm of the gradient squared of a differentiable function $v$ on $M\times F$ with this metric is given by $$|\nabla_{(g_M|_{V_i})\oplus \epsilon^2g_F}v|^2=|\nabla_{g_M}v|_{g_M}^2+\frac{1}{\epsilon^2}|\nabla_{g_F} v|_{g_F}^2.$$ Let $u$ be a smooth function on $\Oe$ vanishing on $\partial\Oe$. Then 
    \begin{align*}
        \int_{\Oe}|\nabla_{\epsilon}u|_{\epsilon}^2dV_{\epsilon}&=(1+O(\epsilon))\sum_{i=1}^K\int_{V_i}\int_F \Big(|\nabla_{g_M}u|_{g_M}^2+\frac{1}{\epsilon^2}|\nabla_{g_F} u|_{g_F}^2\Big)\epsilon^ndV_{g_F}dV_{g_M}\\
        &\geq (1+O(\epsilon))\sum_{i=1}^K\int_{V_i}\frac{1}{\epsilon^2}\Big(\int_F|\nabla_{g_F} u|_{g_F}^2dV_{g_F}\Big)\epsilon^ndV_{g_M}\\
        &\geq (1+O(\epsilon))\frac{1}{\epsilon^2}\lambda_1(F)\sum_{i=1}^K\int_{V_i}\int_F|u|^2\epsilon^ndV_{g_F}dV_{g_M}\\
        &=(1+O(\epsilon))\frac{1}{\epsilon^2}\lambda_1(F)\int_{\Oe}|u|^2dV_{\epsilon}.
    \end{align*} 
    Since the $O(\epsilon)$ term is independent of $u$, it follows from the variational characterization of the Dirichlet eigenvalue problem that $$\lambda_1(\Oe)\geq \frac{1}{\epsilon^2}\lambda_1(F)+O\Big(\frac{1}{\epsilon}\Big).$$ Combining this with the last inequality and the expression for $\lambda_1(p^{-1}(V_i))$ gives $$\lambda_1(\Oe)=(1+O(\epsilon))\lambda_1(p^{-1}(V_i),(g_M|_{V_i})\oplus\epsilon^2 g_F)=(1+O(\epsilon))\frac{\lambda_1(F)}{\epsilon^2}$$ for all $i\in\{1,\dots ,K\}$. Moreover, for each $i$, the same argument used in the proof of Proposition \ref{trivialwarped} implies that $$\lambda_1(p^{-1}(V_i))=(1+O(\epsilon))\frac{\lambda_1(F)}{\epsilon^2}=\lambda_1(\Oe)+O\Big(\frac{1}{\epsilon}\Big).$$  Since $K$ is finite, we may take the $O(\epsilon)$ terms to be independent of $i$. It thus follows from Dirichlet domain monotonicity, Proposition \ref{trivialwarped}, and Lemma \ref{mixedtrivialwarped} that 
    \begin{align*}
        N(\Oe)&\geq \sum_{i=1}^K\#\Big\{j:\lambda_j^{\text{mix}}(V_i)\leq \lambda_1(\Oe)\Big\}\\
        &=\sum_{i=1}^K\#\Big\{j:\lambda_j^{\text{mix}}(V_i)\leq \lambda_1(V_i)+O\Big(\frac1{\epsilon}\Big)\Big\}\\
        &=\sum_{i=1}^{K}N(V_i)+o\Big(\frac{1}{\epsilon^m}\Big)\\
        &=\frac{\omega_m}{(2\pi)^m}\cdot|M|\cdot\Bigg(\sum_{j=1}^{N(F)}\big(\lambda_1(F)-\mu_j(F)\big)^{m/2}\Bigg)\cdot\frac{1}{\epsilon^m}+o\Big(\frac{1}{\epsilon^m}\Big).
    \end{align*}
    Combining this with the upper bound on $N(\Oe)$ above completes the proof. 
\end{proof}

Finally, we show that Theorem \ref{tubenbhd} is a special case of Theorem \ref{asymptoticthm}.

\begin{proof}[Proof of Theorem \ref{tubenbhd}]
    The tubular neighborhood theorem (see, for instance, Theorem 5.25 of Lee \cite{lee}) states that for $\epsilon>0$ sufficiently small, the set of points a distance $\epsilon$ to $M$ forms a fiber bundle over $M$ with fibers diffeomorphic to the ball $B^n$ (in particular, each point in $\Oe$ has a unique point on $M$ to which it is nearest). Fix $x_0\in M$. There exists a neighborhood $U\subseteq M$ of $x_0$ on which there exist functions $h_1,\dots ,h_n:U\to T^{\perp}M$ (the orthogonal complement of $TM\subseteq TX|_M$) such that at each $x\in U$, $(h_1(x),\dots ,h_n(x))$ is an orthonormal basis for $T^{\perp}M$. We thus get a trivializing diffeomorphism $$\phi:U\times B^n\to \Oe,\;\;(x,(y_1,\dots ,y_n))\mapsto \exp_x(\epsilon y_1h_1(x),\dots ,\epsilon y_nh_n(x)),$$ where $\exp_x:T_xX\to X$ is the exponential map. Using elementary facts about exponential maps, we see that $\Oe$ satisfies the hypothesis of Theorem \ref{asymptoticthm}.
\end{proof}

\section*{Appendix}
Here we prove Theorem \ref{funanoestimate}. Since Funano's theorem \cite{funano} proves the same estimate for bounded convex domains with piecewise smooth boundaries, we need only apply a domain approximation argument. \\
\indent Let $S^{n-1}\subseteq \Rbb^n$ denote the unit $(n-1)$-sphere. Let $\epsilon>0$. Let $L,K:S^{n-1}\to(0,\infty)$ be continuous functions with $K(p)>L(p)$ for all $p\in S^{n-1}$ and such that $K(p)-L(p)<\epsilon$. Define 
$$f(L,n,\epsilon):=\begin{cases}
    \displaystyle\sup_{p\in S^{n-1}}\ln\Big(\frac{L(p)+\epsilon}{L(p)}\Big)\;\;&\text{if}\;\;n=2\\ \\\displaystyle\sup_{p\in S^{n-1}}\frac{1}{n-2}\Big(L(p)^{2-n}-(L(p)+\epsilon)^{2-n}\Big)\;\;&\text{if}\;\;n>2
\end{cases}$$
and $$g(L,n,\epsilon):=\sup_{p\in S^{n-1}}\frac{1}{n}\Big((L(p)+\epsilon)^n-L(p)^n\Big).$$ Note that as $\epsilon\to 0^+$ with $L$ and $n$ fixed, both $f$ and $g$ tend toward $0$. \\
\indent Let $$\Omega_{L,K}:=\Big\{rp\in \Rbb^n\mid p\in S^{n-1}, L(p)<r<K(p)\Big\},$$ and let $\nu_1(\Omega_{L,K})$ denote the first eigenvalue of the Laplacian on $\Omega_{L,K}$ with Dirichlet boundary conditions on $L(S^{n-1})$ and Neumann boundary conditions on $K(S^{n-1})$. 
\begin{lem}\label{evalineq2}
    $\nu_1(\Omega_{L,K})\geq \frac{1}{f(L,n,\epsilon)g(L,n,\epsilon)}$. 
\end{lem}
\begin{proof}
    Let $u$ be a $C^1$ function on $\Omega_{L,K}$ that vanishes on $L(S^{n-1})$. Then, using polar coordinates $r\in(0,\infty)$ and $p\in S^{n-1}$ on $\Rbb^n$, 
    \begin{align*}
        |u(r,p)|^2&=\Big|\int_{L(p)}^ru_r(s,p)ds\Big|^2\\
        &=\Big|\int_{L(p)}^r\frac{1}{s^{(n-1)/2}}u_r(s,p)s^{(n-1)/2}ds\Big|^2\\
        &\leq f(L,n,\epsilon)\int_{L(p)}^{K(p)}|\nabla u(s,p)|^2s^{n-1}ds.
    \end{align*}
    Integrating both sides of this inequality gives
    \begin{align*}
        \int_{\Omega}|u(r,p)|^2&\leq \int_{S^{n-1}}\int_{L(p)}^{K(p)}\Bigg(f(L,n,\epsilon)\int_{L(p)}^{K(p)}|\nabla u(s,p)|^2s^{n-1}ds\Bigg)r^{n-1}drdp\\
        &=g(L,n,\epsilon)f(L,n,\epsilon)\int_{\Omega}|\nabla u|^2.
    \end{align*} 
    Since $$\nu_1(\Omega_{L,K})=\inf_{\substack{u\in C^1(\Omega_{L,K})\setminus\{0\}\\ \\u|_{L(S^{n-1})\equiv 0}}}\frac{\int_{\Omega_{L,K}}|\nabla u|^2}{\int_{\Omega_{L,K}}|u|^2},$$ the result follows. 
\end{proof}
Theorem \ref{funanoestimate} is a combination of Proposition 2.3 of \cite{arrieta}, Theorem 1.1 of \cite{funano}, and Lemma \ref{evalineq2}.
\begin{proof}[Proof of Theorem \ref{funanoestimate}]
    Let $\Omega\subseteq\Omega'\subseteq \Rbb^n$ be bounded convex domains. By applying an isometry if necessary, we may suppose without loss of generality that the origin is contained in the interior of $\Omega$. Then the maps $L^{-1}:\partial\Omega\to S^{n-1}$ and $(L')^{-1}:\partial\Omega'\to S^{n-1}$ defined by $L^{-1}(x)=x/|x|$ and $(L')^{-1}(x)=x/|x|$ have continuous inverses $L:S^{n-1}\to \partial\Omega$ and $L':S^{n-1}\to\partial\Omega'$, respectively. \\
    \indent By standard results in convex geometry (see, e.g., Theorem 3.4.1 of \cite{convexbook}), there exists a sequence $\{\Omega_k\}$ (resp. $\{\Omega_k'\}$) of convex domains with smooth boundaries such that $\Omega\subseteq \Omega_k$ (resp. $\Omega'\subseteq\Omega_k'$) for all $k$ and such that for any $x\in \Omega_k$ (resp. $x\in \Omega_k'$), the distance from $x$ to $\Omega$ (resp. $\Omega'$) is less than $1/k$. Rescaling $\Omega_k'$ if necessary, we may also assume that for each $k$, $\Omega_k\subseteq \Omega_k'$. For each $k$, define $K_k:S^{n-1}\to\partial\Omega_k$ (resp. $K_k':S^{n-1}\to\partial\Omega'$) similarly to $L$ (resp. $L'$). 
    By Lemma \ref{evalineq2}, the first mixed eigenvalues $\lambda_1(\Omega_k\setminus \Omega)$ and $\lambda_1(\Omega_k'\setminus\Omega')$ both tend toward $\infty$ as $k\to\infty$. As a result, Proposition 2.3 of \cite{arrieta} implies that for each $j$, $\mu_j(\Omega_k)\to\mu_j(\Omega)$ and $\mu_j(\Omega_k')\to\mu_j(\Omega')$ as $k\to\infty$. Let $E=92^2$ be as in FUnano's theorem. Since $\Omega_k$ and $\Omega_k'$ both have smooth boundary and $\Omega_k\subseteq \Omega_k'$ for each $k$, Theorem 1.1 of \cite{funano} implies that $$\mu_j(\Omega')=\lim_{k\to\infty}\mu_j(\Omega_k')\leq E\cdot n^2\cdot\lim_{k\to\infty}\mu_j(\Omega_k)=E\cdot n^2\cdot \mu_j(\Omega).$$
\end{proof}


\end{document}